\documentclass[twoside,a4paper,11pt]{amsart}

\usepackage{MyStyle}
\begin{document}

\title{The internal Yoneda lemma for locally Cartesian closed $\infty$-categories}
\author{Virgile Constantin}\begin{abstract}
  We formulate and prove internal versions of the Yoneda lemma and of the Yoneda embedding theorem in a finitely complete, locally Cartesian closed $\infty$-category $\C$: for every object $X\in\C$ and every universe $\U$ classifying the diagonal of $X$, the Yoneda map $\yo_X\colon X\to\U^X$ is a monomorphism. The proof uses only finite limits, dependent products and universes, and does not rely on the external Yoneda lemma. The result applies notably to every elementary $\infty$-topos, where it recovers a theorem of Rasekh \cite{Rasekh_YonedaLemmaElementaryTopos}.
\end{abstract}

\maketitle
\tableofcontents
\section*{Introduction}
For a locally small $\infty$-category $\X$, the Yoneda embedding $\yo_\X\colon\X\to\Fun(\X^\op,\Sp)$, $x\mapsto\X(-,x)$, is fully faithful. This is a central result formulated \emph{externally}: it lives in the large $\infty$-category of $\infty$-categories.

\medskip

This paper is concerned with an internal counterpart of this statement, in an $\infty$-category $\C$ that is finitely complete and locally Cartesian closed. Inside such an $\infty$-category, no ambient copy of $\Sp$ is available to receive presheaves. What replaces it is a \emph{universe}: an object $\U\in\C$ that classifies a class of maps in $\C$ and plays, internally, the role that $\Sp$ plays externally \cite{GepnerKock_UnivalenceInLCCC, Rasekh_UnivalenceHigherCategoryTheory}. A map $X\to\U$ encodes a family of objects over $X$, and exponentiation accordingly produces an object $\U^X$ of internal presheaves on $X$. In this setting the objects of $\C$ behave like generalized homotopy types. Indeed, finitely complete locally Cartesian closed $\infty$-categories furnish the categorical semantics of dependent type theory with dependent products and intensional identity types \cite{Kapulkin_LCCQuasicategoriesFromTypeTheory, KapulkinSzumilo_InternalLanguagesOfFinitelyCompleteInftyCategories, Cherradi_InternalLanguagesLCCinfintyCategories}. For an object $X$ whose diagonal is classified by $\U$, one may then ask for an internal analogue of the Yoneda embedding: a Yoneda map $\yo_X\colon X\to\U^X$ (\cref{definition: yoneda map}) that is fully faithful in a suitable internal sense.

\medskip

This paper grew out of the study of higher covering spaces in an $\infty$-topos \cite{Constantin_HigherCoveringSpacesinHigherTopos}, where the identification of the group of deck transformations requires exactly this internal embedding. That application is described in \cref{section: application}.

\medskip

The main result of this paper is the following.
\begin{theorem*}[{Yoneda Embedding, \cref{Yoneda Embedding}}]
    Let $\C$ be a finitely complete locally Cartesian closed $\infty$-category, $X$ an object of $\C$, and $\U$ a universe classifying the diagonal $\Delta_X\colon X\to X\times X$. The Yoneda map \[\yo_X\colon X\to\U^X\] is a monomorphism.
\end{theorem*}
Paralleling the classical case, the Yoneda map $\yo_X$ is defined as the adjoint of the internal mapping presheaf $X(-,-)\colon X\times X\to\U$, namely the classifying map of the diagonal $\Delta_X$. The hypotheses barely exceed what is needed to state the result: finite limits and dependent products produce the object $\U^X$ together with the internal path and equivalence objects, and the universe supplies the target of the Yoneda map. Every elementary $\infty$-topos in the sense of Rasekh \cite{Rasekh_TheoryElementaryHigherToposes} has, for each of its objects $X$, a universe classifying $\Delta_X$ (it has \emph{sufficient universes}), so the theorem applies to all of its objects. As a corollary, one obtains a Yoneda embedding theorem for small $\infty$-groupoids:
\begin{corollary*}[{\cref{Cor: yoneda for groupoids}}]
    Let $\X$ be a small $\infty$-groupoid. Then the Yoneda map $\yo_\X\colon\X\to \Psh(\X)$ is fully faithful.
\end{corollary*}

\medskip

The defining property of a universe is univalence (\cref{Univalence Axiom}), stipulating that paths in $\U$ correspond to equivalences. One could take this as the definition of a universe, as is common in the type-theoretic literature. Instead, we prefer to derive it from the classifying property (\cref{prop: a universe is an object of equivalences}) for three reasons. First, the derivation avoids invoking the external Yoneda lemma to identify $\U$ with the desired object of equivalences. Second, it keeps our definition of universes aligned with the existing one in the non-presentable setting \cite{Rasekh_UnivalenceHigherCategoryTheory}. Lastly, it forces the fibrational route that the rest of the paper follows, as explained in the next paragraph. Accordingly, the external fibrational picture is developed first, in \cref{subsection: fibrations}: the preliminary results rest on it, and it fixes the intuition that the internal arguments of \cref{section : Yoneda} mirror.

\medskip

\textbf{Method.} The main idea is the fibrational dictionary, taken internally. A family over a base $X$, called an \textit{internal fibration} $E\to X$, is encoded by a map $X\to\U$ into a universe, an \textit{internal presheaf} (\cref{terminology: internal fibrations and presheaves}). For instance, the representable presheaf $X(-,x)\colon X\to\U$ encodes the representable fibration $\slice{X}{x}\to X$ (\cref{subsection: representable fibrations}). The Yoneda lemma is then stated internally using the fibrational language: as a correspondence between maps of fibrations $\slice{X}{x}\to P$ over $X$ and the fiber $P_x$. The naive formulation using internal presheaves fails, forcing the fibrational approach (\cref{subsection: yoneda lemma}). On this basis the Yoneda map is built from the classifying map of $\Delta_X$, the internal analogue of the domain-codomain fibration, and \textit{fully faithful} is rendered as \textit{monomorphism} (\cref{subsection: the Yoneda map}). We never use the Yoneda lemma: neither the formula $\mathrm{Nat}(\yo c,F)\simeq F(c)$ nor the full faithfulness of $\yo\colon\C\to\Psh(\C)$. Two of its classical consequences are needed -- conservativity and essential injectivity: a map is an equivalence as soon as it is one on all representables, and equivalent representables have equivalent representing objects -- and both are proved directly, by fibrational means, in \cref{prop: equivalence is detected on representables}.

\medskip

\textbf{Organization of the paper.}
\cref{section : preliminaries} fixes conventions and assembles the preliminaries: the fibrational dictionary, monomorphisms, local Cartesian closure, objects of equivalences, and universes. The Yoneda map is constructed, and proved to be a monomorphism, in \cref{section : Yoneda}. \cref{section: application} recovers the classical Yoneda statements for small $\infty$-groupoids and presents the application to higher covering spaces that motivated this work.

\medskip

\textbf{Relation to other work.}
\cref{Yoneda Embedding} was proved by Rasekh \cite{Rasekh_YonedaLemmaElementaryTopos} under the stronger hypotheses of an elementary $\infty$-topos, by a different route, analyzing the Yoneda map through the external mapping spaces $\Map(A,\yo_X)$ for $A\in\C$. In the type-theoretic literature, the Yoneda lemma appears in \cite[Theorem~13.3.3]{Rijke_IntroductionHoTT} and the Yoneda embedding in \cite[\S27]{Escardo_IntroductionUnivalentFoundations}. These are the type-theoretic counterparts of our results; the correspondence is made precise by the internal-language theorems cited above. Our proofs are independent of that route: they take place directly in the $\infty$-category, at a generality that presupposes no type-theoretic presentation. Univalence of universes was established by Gepner and Kock \cite{GepnerKock_UnivalenceInLCCC} for presentable locally Cartesian closed $\infty$-categories, and by Rasekh \cite{Rasekh_UnivalenceHigherCategoryTheory} in the non-presentable setting. The Yoneda lemma has been proven in a variety of contexts, for instance: \cite{Martini_YonedasLemmaInternalHigherCategories} for internal $\infty$-categories, \cite{frohlichMoser_YonedaLemmaRepresentationTheorem} for double categories, \cite{Rasekh_YonedaForSimplicialSpaces} for simplicial spaces, \cite{Barwick_ParametrizedHigherCategoryTheory} for parametrized $\infty$-categories, or \cite{RiehlShulman_TypeTheoryForSynteticInftyCategories} for synthetic $\infty$-categories in simplicial type theory.
\medskip

\textbf{Use of AI assistance.}
The writing of this paper was assisted by Claude, a large language model made by Anthropic. Mathematically, the AI identified a gap in an earlier version of the proof of \cref{Yoneda Embedding} (the comparison map induced by $\yo_X$ was not shown to coincide with the equivalence produced by the three main lemmas) and suggested the retract rigidity argument (\cref{lemma: retract rigidity}) to close it. It also observed that the hypotheses of an elementary $\infty$-topos could be weakened to those of the present paper, by replacing the proof of \cref{prop: an equivalence that factors with a mono is an equivalence}, which was originally based on effective epimorphisms, with the elementary argument of \cref{lemma: mono with section is an equivalence}. Editorially, it suggested trimming certain passages for conciseness, and assisted with English phrasing and overall smoothness of the exposition. All AI-produced suggestions were verified, and where necessary corrected, by the author.

\section{Preliminaries}\label{section : preliminaries}
This section fixes notation and assembles the tools used in \cref{section : Yoneda}. The presentation obeys one constraint: we never invoke the $\infty$-categorical Yoneda lemma. Its role is taken over by the detection of equivalences on representables (\cref{prop: equivalence is detected on representables}), which shows that the Yoneda map $\C\to\Psh(\C)$ is conservative and essentially injective. This comes as an easy consequence of the fibrational dictionary set up in \cref{subsection: fibrations}. The remaining subsections supply the setting in which this dictionary is carried inside a finitely complete locally Cartesian closed $\infty$-category $\C$: monomorphisms (\cref{subsection: monomorphisms}), local Cartesian closure (\cref{subsection: LCCC}), the internal object of equivalences (\cref{subsection: equivalences}), and universes (\cref{subsection: universes}). A reader familiar with this material may proceed to \cref{section : Yoneda} and refer back as needed.
\subsection{Conventions}
\paragraph{\textbf{$\infty$-categorical}}
Throughout this paper, we employ the language of higher category theory, considering only homotopy-invariant constructions. By an $\infty$-category, we mean an $(\infty, 1)$-category. We fix a hierarchy of Grothendieck universes and speak of small, large, and very large $\infty$-categories. Unless stated otherwise our $\infty$-categories are locally small, the exception being $\hat{\Sp}$. Our arguments are essentially model-independent, as they rely solely on the manipulation of finite limits, adjunctions, internal homs, and universes. To be precise, however, one could interpret all our results within the framework of \emph{quasi-categories} as developed by Joyal \cite{Joyal_QuasiCategoriesAndKanComplexes} and Lurie \cite{LurieHTT}, which provides a complete and rigorous theory of $\infty$-categories. Other references are \cite{Cisinski_HigherCategoriesAndHomotopicalAlgebra, RiehlVerity_ElementsOfInftyCategoryTheory}.

\medskip

\paragraph{\textbf{Notational}}
We adhere to the following terminological and notational conventions. We use the word \emph{space} to refer to an $\infty$-groupoid or abstract homotopy type, and denote the large $\infty$-category of small spaces by $\mathcal{S}$ and the very large $\infty$-category of large
spaces by $\hat{\Sp}$. An arrow $f\colon X\to Y$ is called an \emph{equivalence} if it has a two-sided inverse $g\colon Y\to X$ such that $gf\simeq 1_X$ and $fg\simeq 1_Y$. The adjective \emph{unique} is always understood to mean unique up to a contractible space of choices. A diagram \emph{commutes} if there are coherent homotopies filling all of its simplices, specified as part of the data. We follow the standard convention of leaving these homotopies implicit when they are uniquely determined (as in the case of compositions) or provided by a universal property (as in the case of limits). All categorical constructions, including limits, mapping objects, and adjunctions, are to be understood in their $\infty$-categorical (homotopical) sense. When necessary, we implicitly view ordinary 1-categories as $\infty$-categories via the nerve embedding. A terminal object is denoted by $\ast$. For any two objects $X$ and $Y$ in a locally small $\infty$-category $\C$, we write $\C(X,Y)\in\mathcal{S}$ for the space of maps between $X$ and $Y$. A commutative square 
\begin{equation}\label{eq: Cartesian square}
  \begin{tikzcd}
	P & X \\
	Y & Z
	\arrow[from=1-1, to=1-2]
	\arrow[from=1-1, to=2-1]
	\arrow[from=1-2, to=2-2]
	\arrow[from=2-1, to=2-2, "f"]
\end{tikzcd}
\end{equation}
is \textit{Cartesian} if $P$ is the pullback of the diagram $Y\to Z\leftarrow X$. For $a,b\colon T\to X$ generalized points, the \textit{path object} $X(a,b)$ in $\slice{\C}{T}$ is the generalized fiber
\begin{equation}\label{definition: path objects}
    \begin{tikzcd}
    {X(a,b)} \arrow[r] \arrow[d] \arrow[dr, phantom, "\lrcorner"{anchor=center, pos=0.125}] & X \arrow[d, "\Delta"] \\
    T \arrow[r, "{(a,b)}"'] & {X \times X}
\end{tikzcd}.
\end{equation}
Note that by a Fubini argument, the path object of a single point $a\colon T\to X$ is the loop object $X(a,a)\simeq\Omega_aX$ in $\slice{\C}{T}$.

We record a few basic categorical properties for future reference:
\begin{enumerate}
    \item\label{item: Hom C vs slice} For $Y,Z\in \C$ and $X\in\slice{\C}{Z}$, there is an equivalence of homotopy types $\C(X,Y)\simeq\slice{\C}{Z}(X,Y\times Z)$.
    \item\label{item: pullback of a product with X} For any map $X\to A$, the square \eqref{eq: Cartesian square} is Cartesian if and only if the following induced square is Cartesian:
\[\begin{tikzcd}
	P & X \\
	{Y\times A} & {Z\times A}
	\arrow[from=1-1, to=1-2]
	\arrow[from=1-1, to=2-1]
	\arrow[from=1-2, to=2-2]
	\arrow[from=2-1, to=2-2, "f\times A"]
\end{tikzcd}.\]
\item\label{item: Slice of slice} Let $f\colon X\to Y$ be a map. There is an equivalence of slice $\infty$-categories $\slice{(\slice{\C}{Y})}{f}\simeq\slice{\C}{X}$.
\end{enumerate}
Finally, we record a rigidity property of slices admitting a retraction:
\begin{lemma}\label{lemma: retract rigidity} 
    Let $\C$ be an $\infty$-category and $i\colon A\to S$ a morphism admitting a retraction $r\colon S\to A$ (so $r\circ i\simeq 1_A$). Let $B\in\slice{\C}{S}$ and suppose there exists an equivalence $\varphi\colon A\xrightarrow{\sim}B$ in $\slice{\C}{S}$. Then every morphism $A\to B$ in $\slice{\C}{S}$ is an equivalence.
\end{lemma}
\begin{proof}
    Let $f\colon A\to B$ be a morphism in $\slice{\C}{S}$, and set $u\coloneq \varphi ^{-1}\circ f\colon A\to A$, an endomorphism of $A$ over $S$. Its underlying morphism comes equipped with a homotopy $i\circ u\simeq i$. Whiskering with $r$ gives \[u\simeq (r\circ i)\circ u\simeq r\circ(i\circ u)\simeq r\circ i\simeq 1_A.\] Hence $u$ is an equivalence, and therefore so is $f\simeq\varphi\circ u$.
\end{proof}
\subsection{Fibrations}\label{subsection: fibrations}
Right fibrations over $\C$ encode presheaves $\C^{op}\to \hat{\Sp}$ and bifibrations over $\C\times\D$ encode functors $\C^{op}\times\D\to\hat{\Sp}$. The corresponding equivalences of $\infty$-categories \cite[Theorem~2.2.1.2]{LurieHTT}: 
\begin{equation}\label{eq: (un)straightening}
    \mathscr{R}\mathrm{Fib}(\C)\simeq\Fun(\C^\op,\hat{\Sp}),\qquad\mathscr{B}\mathrm{Fib}(\C\times\D)\simeq\Fun(\C^\op\times\D,\hat{\Sp})
\end{equation}
are the \emph{straightening} equivalences, their inverses are the \emph{unstraightening} equivalences. We explain the typical examples which will guide our internal approach. We denote by $\Orb_\C\coloneq\Fun([1],\C)$ the arrow $\infty$-category. If $S$ is a class of maps in $\C$, we denote by $\Orb_\C^S\subseteq\Orb_\C$ the full subcategory on $S$. We write $\Orb_\C^{(S)}$ for the wide subcategory of $\Orb_\C^S$ whose morphisms are the Cartesian squares.
\begin{example}\label{examples: fibrations}
\begin{enumerate}[label=\arabic*., ref=\thelemma.\arabic*]
    \item There is a bifibration $(\mathrm{dom,cod})\colon\Orb_\C\to\C\times \C$ given by the domain and codomain fibrations \cite[Cor.~2.4.7.11]{LurieHTT}. This bifibration corresponds to the mapping space presheaf $\C(-,-)\colon\C^\op\times\C\to\Sp\subseteq \hat{\Sp}$.
    \item\label{examples: fibrations 2} Let $X\in\C$ be an object. The slice $\infty$-category over $X$ is obtained as the following pullback of $\infty$-categories:
\[\begin{tikzcd}
	{\slice{\C}{X}}\arrow[dr, phantom, "\lrcorner"{anchor=center, pos=0.125}] \arrow[r] \arrow[d, "{\Sigma_X}"'] & {\Orb_\C} \arrow[d] \\
	\C \arrow[r, "{(1_\C,\cst_X)}"'] & {\C\times\C}
\end{tikzcd}.\]
The left vertical map $\Sigma_X\colon \slice{\C}{X}\to\C$ is a right fibration \cite[Cor.~2.1.2.2]{LurieHTT}, and corresponds to the representable presheaf $\C(-,X)\colon \C^\op\to\Sp\subseteq \hat{\Sp}$. If $f\colon X\to Y$ is a map, one obtains a functor $\Sigma_f\colon\slice{\C}{X}\to\slice{\C}{Y}$, called the \textit{dependent sum}, by working in the slice over $Y$ and using \prelref{item: Slice of slice}.
\item The composition \[\mathrm{cod}\colon \Orb_\C^{(\mathrm{all})}\to \Orb_\C\xrightarrow{\mathrm{cod}}\C\] is a right fibration \cite[6.1.3.4]{LurieHTT} corresponding to the presheaf $(\slice{\C}{-})^\simeq\colon \C^{op}\to \hat{\Sp}$.
\end{enumerate}
\end{example}
\begin{definition}\label{def: representable fibration}
    A right fibration $\mathscr{R}\to\C$ is \textit{representable} if it is equivalent to $\slice{\C}{X}\to\C$ for some $X\in\C$, called the \textit{representing} object.
\end{definition}
Representable fibrations thus correspond precisely to representable functors. Here is a simple criterion to recognize representable fibrations.
\begin{prop}[{\cite[Prop.~4.4.4.5]{LurieHTT}}]\label{prop: representable right fibrations and terminal objects}
    A right fibration $p\colon \mathscr{R}\to\C$ is equivalent to $\slice{\C}{X}$ if and only if $\mathscr{R}$ has a terminal object. In that case $X$ is the image $p(\ast_{\mathscr{R}})$ of the terminal object.
\end{prop}
The fibrational point of view yields conservativity and essential injectivity of the Yoneda functor $\C\to\Psh(\C)$, which replaces the Yoneda lemma throughout the paper.
\begin{prop}\label{prop: equivalence is detected on representables}
    Let $\C$ be an $\infty$-category.
    \begin{enumerate}[label=\arabic*., ref=\thetheorem.\arabic*]
        \item\label{detection: maps} A morphism $f\colon X\to Y$ is an equivalence if and only if $\Sigma_f\colon \slice{\C}{X}\to\slice{\C}{Y}$ is an equivalence. Equivalently, if and only if $\C(T,f)\colon\C(T,X)\to\C(T,Y)$ is an equivalence for every $T\in\C$.
        \item\label{detection: objects} Two objects $X,Y$ are equivalent if and only if the right fibrations $\slice{\C}{X}$ and $\slice{\C}{Y}$ are equivalent over $\C$, equivalently if and
        only if $\C(-,X)\simeq\C(-,Y)$ as functors $\C^\op\to\Sp$.
    \end{enumerate}
\end{prop}
\begin{proof}
    For (1): Suppose $\Sigma_f\colon\slice{\C}{X}\to\slice{\C}{Y}$ is an equivalence of right fibrations. Then it preserves terminal objects, so $f\simeq\Sigma_f(1_X)\simeq 1_Y$ is an equivalence. The converse is clear.
    
    For (2): Suppose that there is an equivalence of right fibrations $\slice{\C}{X}\simeq\slice{\C}{Y}$ over $\C$. By \cref{prop: representable right fibrations and terminal objects}, the image $\Sigma_X(1_X)=X$ must be equivalent to $Y$. The converse is Proposition \ref{detection: maps} applied to an equivalence $X\to Y$.
    \end{proof}
\subsection{Monomorphisms}\label{subsection: monomorphisms}
\begin{definition}
    Let $\C$ be an $\infty$-category. An object $X\in \C$ is \textit{$(-1)$-truncated} if for every other object $Z\in \C$, the mapping space $\C(Z,X)$ is empty or contractible. A map $f\colon X\to Y$ is \textit{$(-1)$-truncated} if it is so as an object of the slice $\slice{\C}{Y}$. 
\end{definition}
\begin{notation}
    A $(-1)$-truncated object is also called \textit{subterminal}, and a $(-1)$-truncated map a \textit{monomorphism}. 
\end{notation}
\begin{prop}\label{prop: characterization monomorphisms}
Let $f\colon X\to Y$ be a map of an $\infty$-category $\C$. The following are equivalent:
\begin{enumerate}[label=\arabic*., ref=\thetheorem.\arabic*]
    \item $f$ is a monomorphism;
    \item\label{prop: characterization monomorphisms 2} the functor $\Sigma_f\colon \slice{\C}{X}\to\slice{\C}{Y}$ is fully faithful;
    \item\label{prop: characterization monomorphisms 3} The functor $\slice{\C}{X}\to\slice{\C}{X}\times_{\slice{\C}{Y}}\slice{\C}{X}$ is an equivalence.
\end{enumerate}
If the pullback $X\times _YX$ exists, then those are equivalent to:
\begin{enumerate}[resume, label=\arabic*., ref=\thetheorem.\arabic*]
    \item\label{prop: characterization monomorphisms 4} the diagonal $X\to X\times_YX$ is an equivalence;
    \item\label{prop: characterization monomorphisms 5} the following commutative square is Cartesian:
\[\begin{tikzcd}
	X & Y \\
	{X\times X} & {Y\times Y}
	\arrow["f", from=1-1, to=1-2]
	\arrow["{\Delta_X}"', from=1-1, to=2-1]
	\arrow["{\Delta_Y}", from=1-2, to=2-2]
	\arrow["{f\times f}", from=2-1, to=2-2]
\end{tikzcd}.\]
\end{enumerate}
\end{prop}
\begin{proof}
    The only step that would ordinarily use Yoneda is $(3)\Leftrightarrow(4)$: $\Delta_f\colon X\to X\times_Y X$ is an equivalence if and only if $\slice{\C}{X}\to\slice{\C}{X\times_YX}$ is an equivalence. This holds by the detection of equivalences on representables (Proposition \ref{detection: maps}).
\end{proof}
We record two lemmas detecting when a monomorphism is an equivalence.
\begin{lemma}\label{lemma: mono with section is an equivalence}
    Let $\C$ be an $\infty$-category. A monomorphism $f\colon A\to X$ in $\C$ admitting a section is an equivalence.
\end{lemma}
\begin{proof}
    Let $g\colon X\to A$ satisfy $fg\simeq 1_X$. The homotopy fiber of $\C(A,f)\colon\C(A,A)\to\C(A,X)$ over $f$ is the mapping space $\slice{\C}{X}(f,f)$, a point of it being a map $A\to A$ together with a homotopy witnessing that it lies over $f$. It is inhabited by $1_f$, and also by $gf$, via $f(gf)\simeq(fg)f\simeq f$. Since $f$ is $(-1)$-truncated, $\slice{\C}{X}(f,f)$ is contractible, so $gf\simeq 1_A$.
\end{proof}
\begin{prop}\label{prop: an equivalence that factors with a mono is an equivalence}
    Let $\C$ be an $\infty$-category, and $X\xrightarrow{f}Y\xrightarrow{g}Z$ be maps in $\C$ such that $g$ is a monomorphism and $g\circ f$ is an equivalence. Then $f$ and $g$ are equivalences. 
\end{prop}
\begin{proof}
   The composition $f\circ (gf)^{-1}$ is a section of $g$. By \cref{lemma: mono with section is an equivalence}, $g$ is an equivalence. By two-out-of-three, $f$ is one as well.
\end{proof}
\subsection{Locally Cartesian closed $\infty$-categories}\label{subsection: LCCC}
Let $f\colon X\to Z$ be a map in an $\infty$-category and consider the dependent sum functor $\Sigma_f\colon \slice{\C}{X}\to\slice{\C}{Z}$. If $\C$ has pullbacks, the pullback functor $f^\ast\colon \slice{\C}{Z}\to\slice{\C}{X}$ is right adjoint to $\Sigma_f$.
\begin{definition}
    An $\infty$-category $\C$ is \textit{locally Cartesian closed} if the pullback functor $f^\ast$ has a right adjoint for all maps $f\colon X\to Z$ in $\C$.
\end{definition} 
When that is the case, the right adjoint is called the \textit{dependent product}, denoted $\prod_f$, yielding a triple of adjunctions:
\[\begin{tikzcd}
	{\slice{\C}{X}} && {\slice{\C}{Z}}
	\arrow[""{name=0, anchor=center, inner sep=0}, "{\Sigma_f}", curve={height=-24pt}, from=1-1, to=1-3]
	\arrow[""{name=1, anchor=center, inner sep=0}, "{\prod_f}"', curve={height=24pt}, from=1-1, to=1-3]
	\arrow[""{name=2, anchor=center, inner sep=0}, "{f^\ast}"{description}, from=1-3, to=1-1]
	\arrow["\dashv"{anchor=center, rotate=-90}, draw=none, from=0, to=2]
	\arrow["\dashv"{anchor=center, rotate=-90}, draw=none, from=2, to=1]
\end{tikzcd}.\]
Let $f\colon X\to Z$ be an object of the slice $\slice{\C}{Z}$. The product functor $-\times f$ can be decomposed as the composition $\slice{\C}{Z}\xrightarrow{f^\ast}\slice{\C}{X}\xrightarrow{\Sigma_f}\slice{\C}{Z}$. When $\C$ is locally Cartesian closed, both functors have right adjoints, yielding a right adjoint $(-)^f$ to the product functor. In particular each slice is Cartesian closed. We write $\Map(X,Y)=Y^X$  (resp. $\Map_{/Z}(X,Y)$) for the internal hom object in $\C$ (resp. $\slice{\C}{Z}$). The following lemma relates dependent products and pullback functors arising from a Cartesian square.
\begin{lemma}[{Beck-Chevalley Condition, \cite[Lemma 2.1]{Frey&Rasekh_ConstructingCoproductsInLCCC}}]
    Let $\C$ be a locally Cartesian closed $\infty$-category with pullbacks, and let 
\[\begin{tikzcd}
	P \arrow[r, "h"]\arrow[dr, phantom, "\lrcorner"{anchor=center, pos=0.125}] \arrow[d, "k"'] & A \arrow[d, "f"] \\
	B \arrow[r, "g"'] & C
\end{tikzcd}\]
be a Cartesian square. The canonical natural transformation \[f^\ast\circ \prod_g\xrightarrow[]{\simeq}\prod_h\circ k^\ast\]
is an equivalence.
\end{lemma}

\subsection{Equivalences}\label{subsection: equivalences}
 When $\C$ is locally Cartesian closed and finitely complete, internal homs have subobjects of equivalences $\Eq(X,Y)\subseteq \Map(X,Y)$, representing the subfunctor $\C^\op\to\Sp$
\[T \mapsto \slice{\C}{T}^\simeq(X\times T,Y\times T)\subseteq \slice{\C}{T}(X\times T, Y\times T).\] 
The object of equivalence can be constructed explicitly as a finite limit of objects built from internal hom objects \cite[Lemma 2.8]{Vergura_LocalizationInTopos}. Although Vergura works in a presentable setting, that hypothesis is never used in the construction. We denote by $\Eq_{/Z}(X,Y)\subseteq\Map_{/Z}(X,Y)\in\slice{\C}{Z}$ the objects of equivalences in the slice over $Z$. Lastly, objects of equivalences are stable under pullback: if $f\colon T\to Z$ is a map, and $X,Y$ are two objects in the slice over $Z$, then 
\begin{equation}\label{eq: object of equivalences is stable under pullback}
    f^\ast\left(\Eq_{/Z}(X,Y)\right)\simeq \Eq_{/T}(f^\ast X,f^\ast Y).
\end{equation}
\subsection{Universes}\label{subsection: universes}
Our exposition differs slightly from the usual sources, again to avoid the Yoneda lemma. The crucial property of a universe is univalence (\cref{Univalence Axiom}), which we obtain immediately from \cref{prop: a universe is an object of equivalences}. That proposition therefore becomes our target. It is proved by Rasekh \cite[Lemma~2.12]{Rasekh_UnivalenceHigherCategoryTheory}, but his argument uses the Yoneda lemma to construct the comparison map \[\slice{\C}{\U}\to\Orb_\C^{(\mathrm{all})}.\] This is avoidable: we construct that map only when $p$ is univalent, directly from the univalence hypothesis rather than via
Yoneda, after which Rasekh's propositions apply verbatim to yield the result.
\begin{definition}[{\cite[Def. 2.1]{Rasekh_UnivalenceHigherCategoryTheory}}]\label{def: univalent map}
Let $\C$ be a finitely complete $\infty$-category. A map $p\colon \U_\ast\to\U$ is \textit{univalent} if it is subterminal in $\Orb_\C^{(\mathrm{all})}$. The subcategory of univalent maps is denoted $\U\mathrm{niv}_\C\coloneq\left(\Orb_\C^{(\mathrm{all})}\right)^{\sleq{-1}}$.
\end{definition}
A univalent map is also called a \textit{universal fibration}, and $\U$ a \textit{universe}.
  \begin{prop}\label{prop: characterization of univalent map}
      Let $p\colon \U_\ast\to \U$ be a map in $\C$. Then the following are equivalent:
      \begin{enumerate}
          \item $p$ is univalent;
          \item The projection $\Sigma_p\colon \slice{\Orb_\C^{(\mathrm{all})}}{p}\to \Orb_\C^{(\mathrm{all})}$ is fully faithful.
      \end{enumerate}
      In that case, there exists a fully faithful functor $\slice{\C}{\U}\to\Orb_\C^{(\mathrm{all})}$.
  \end{prop}
  \begin{proof}
      The equivalence of the two points is Proposition \ref{prop: characterization monomorphisms 2}. Now consider the following composition of right fibrations:
      \[\slice{\Orb_\C^{(\mathrm{all})}}{p}\xrightarrow{\Sigma_p}\Orb_\C^{(\mathrm{all})}\xrightarrow{\mathrm{cod}}\C.\]
      Because right fibrations are stable under composition, the composition is a right fibration, and since the domain has a terminal object, namely $1_p$ with image $\U$ in $\C$, there is an equivalence \[\slice{\C}{\U}\simeq \slice{\Orb_\C^{(\mathrm{all})}}{p}.\]
      It follows that $\Sigma_p$ becomes a map $\slice{\C}{\U}\to\Orb_\C^{(\mathrm{all})}$. This map is fully faithful by the first part of the proof.
  \end{proof}
  \begin{definition}
      We say that $p\colon\U_\ast\to\U$ \textit{classifies} a map $f\colon Y\to X$ if $f$ is in the essential image of $\Sigma_p\colon \slice{\Orb_\C^{(\mathrm{all})}}{p}\to \Orb_\C^{(\mathrm{all})}$.
  \end{definition}
  
  The fiber of the map $\slice{\C}{\U}\to\Orb_\C^{(\mathrm{all})}$ of right fibrations over $\C$ at an object $X\in\C$ yields a fully faithful functor between $\infty$-groupoids
  \[\C(X,\U)\to \slice{\C}{X}^\simeq\]
  given by pulling back along $p$. Intuitively, this means that families $\corner{E}\colon X\to\U$ correspond to certain objects $E\to X$, and paths $\corner{E}\simeq \corner{E'}$ in $\U$ correspond exactly to equivalences $E \simeq E'$ over $X$, and so on. This is an internal straightening-unstraightening equivalence on its image. Because our discussion of the Yoneda lemma will follow the fibrational analogy, both the horizontal incarnation of a family of objects, as a map into the universe, and its vertical counterpart, as a fibered object over the base, will be used, so we begin by giving them explicit terminology. 
\begin{definition}\label{terminology: internal fibrations and presheaves}
    Let $f\colon Y\to X$ be a map classified by a universe $\U$. We call the classifying map $X\to\U$ an \textit{internal presheaf}, while $f$ is an \textit{internal fibration}.
\end{definition}
\begin{prop}[{\cite[Lemma~2.12]{Rasekh_UnivalenceHigherCategoryTheory}}]\label{prop: a universe is an object of equivalences}
    Let $p\colon\U_\ast\to\U$ be a univalent map in a finitely complete locally Cartesian closed $\infty$-category $\C$. Then there is an equivalence over $\U\times\U$:
    \begin{equation*}
        \begin{tikzcd}[ampersand replacement=\&]
        \U \&\& {\Eq_{/\U\times\U}(\U_\ast\times\U,\U\times\U_\ast)} \\
        \& {\U\times\U}
        \arrow[from=1-1, to=1-3, "\simeq"]
        \arrow[from=1-1, to=2-2, "\Delta"]
        \arrow[from=1-3, to=2-2]
    \end{tikzcd}
       \end{equation*}
\end{prop}
\begin{proof}
 Let $\slice{\C}{\U}\to\Orb_\C^{(\mathrm{all})}$ be the fully faithful functor constructed in \cref{prop: characterization of univalent map}. Equivalently, the diagonal $\slice{\C}{\U}\to\slice{\C}{\U}\times_{\Orb_\C^{(\mathrm{all})}}\slice{\C}{\U}$ is an equivalence. In the slice over $\U\times\U$, this is an equivalence of right fibrations, where the second is identified in \cite[Lemma~2.12]{Rasekh_UnivalenceHigherCategoryTheory} with the one represented by the object of equivalences $\Eq_{/\U\times\U}(\U_\ast\times\U,\U\times\U_\ast)$, whence the claim.
\end{proof}
This intuitive property makes internal univalence essentially a tautology. It stipulates that paths in a universe $\U$ between objects $P,Q\in\C$ correspond to equivalences $P\simeq Q$. 

\begin{prop}[Univalence]\label{Univalence Axiom}
    Let $\C$ be a finitely complete and locally Cartesian closed $\infty$-category. Let $P,Q$ be internal fibrations over $T$, classified by internal presheaves $\corner{P}, \corner{Q} \colon T \to \U$ into a univalent universe $\U$. Then there is a Cartesian square in $\C$:
\[\begin{tikzcd}
    {\Eq_{/T}(P,Q)} \arrow[r] \arrow[d] \arrow[dr, phantom, "\lrcorner"{anchor=center, pos=0.125}] & \U \arrow[d, "\Delta_{\U}"] \\
    T \arrow[r, "{(\corner{P},\corner{Q})}"'] & {\U \times \U}
\end{tikzcd}.\]
   In other words, there is an equivalence $\U(\corner{P},\corner{Q})\simeq\Eq_{/T}(P,Q)$.
\end{prop}
\begin{proof}
    By \cref{prop: a universe is an object of equivalences}, $\U$ is the object of equivalences over $\U\times \U$. Objects of equivalences are stable under pullback by \eqref{eq: object of equivalences is stable under pullback}, hence we are done.
\end{proof}
\begin{example}\label{example: universe in spaces}
    In $\hat{\Sp}$, the forgetful functor $\Sp^\simeq_\ast\to\Sp^\simeq$ is univalent. By definition a path, i.e. a morphism, is precisely an equivalence, which illustrates \cref{Univalence Axiom}.
\end{example}

\section{The Yoneda embedding}\label{section : Yoneda}
In this section, we formulate and prove an internal version of the Yoneda embedding within a finitely complete locally Cartesian closed $\infty$-category. The strategy parallels the classical one: we first establish an internal Yoneda lemma (\cref{Yoneda Lemma}), and then deduce that the Yoneda map is a monomorphism by applying it to representables. We start in \cref{subsection: the Yoneda map} by defining the Yoneda map. To state the Yoneda lemma in \cref{subsection: yoneda lemma}, we need the notion of internal representable presheaves and fibrations, which we define in \cref{subsection: representable fibrations}. Finally the proof of the Yoneda embedding theorem is in \cref{subsection: proof of Yoneda}.

\subsection{The Yoneda map}\label{subsection: the Yoneda map}
We first explain how to define the Yoneda map internally, by translating the classical construction step by step. The following discussion, which is not used in any proof, constitutes motivation in the $\infty$-category $\Sp$ of homotopy types. Classically, for a locally small $\infty$-category $\X$, it is defined as the adjoint of the mapping space functor $\X(-,-)\colon \X^\op\times\X\to \Sp$. When $\X$ is a small $\infty$-groupoid, three simplifications occur:
\begin{enumerate}
    \item the functor factors through the maximal subgroupoid $\Sp^\simeq\subset\Sp$;
    \item the functor takes values in the $\infty$-category $\Sp_\kappa$ of $\kappa$-small homotopy types, for any sufficiently large regular cardinal $\kappa$;
    \item $\X$ is equivalent to its opposite category $\X^\op$.
\end{enumerate}
It follows that the mapping space functor simplifies to a map $\X(-,-)\colon \X\times\X\to \Sp_\kappa^\simeq$ between small $\infty$-groupoids, that is, a morphism in the $\infty$-category $\Sp$. Via the straightening-unstraightening correspondence (\cref{subsection: fibrations}), this functor is classified by the bifibration $(\mathrm{dom,cod})\colon \Orb_\X\to \X\times\X$. Since $\X$ is a groupoid, $\Orb_\X\simeq \Fun(I,\X)$, where $I$ is the walking isomorphism $1$-category. Since $I$ is contractible, $\Orb_\X\simeq\X$ and we therefore obtain a Cartesian square in $\Sp$:
\[\begin{tikzcd}
    {\X} \arrow[r] \arrow[d, "\Delta"'] \arrow[dr, phantom, "\lrcorner"{anchor=center, pos=0.125}] & {\Sp_{\kappa,\ast}^\simeq} \arrow[d] \\
    {\X \times \X} \arrow[r, "{\X(-,-)}"'] & {\Sp_\kappa^\simeq}
\end{tikzcd}.\]

The conclusion of this discussion is the following dictionary. In the $\infty$-category $\Sp$ of homotopy types, the mapping space functor $\X(-,-)\colon \X\times \X\to\Sp^\simeq_\kappa$ of an $\infty$-groupoid $\X$ is a classifying map for the diagonal $\Delta\colon \X\to \X\times \X$, the fiber of $\Delta$ over $(x,y)$ being the path space $\X(x,y)$. We now pass to the general setting in which every ingredient of this description has an internal counterpart: a finitely complete, locally Cartesian closed $\infty$-category equipped with a universe $\U$ classifying the diagonal $\Delta_X$ of an object $X$. Here $X$ plays the role of $\X$, and $\U$ that of $\Sp^\simeq_\kappa$. There is then a map $X(-,-)\colon X\times X\to\U$, unique up to a contractible space of choices (\cref{def: univalent map}), such that the following square is Cartesian:
\[\begin{tikzcd}
    X \arrow[r] \arrow[d, "\Delta"'] \arrow[dr, phantom, "\lrcorner"{anchor=center, pos=0.125}] & {\U_\ast} \arrow[d] \\
    {X \times X} \arrow[r, "{X(-,-)}"'] & \U
\end{tikzcd}.\]
\begin{definition}\label{definition: yoneda map}
   The \textit{Yoneda map} $\yo_X\colon X\to\U^X$ is defined as the adjoint of $X(-,-)\colon X\times X\to\U$.
\end{definition}
Classically, the Yoneda embedding of an $\infty$-groupoid is fully faithful, and a functor between $\infty$-groupoids is fully faithful precisely when it is a monomorphism of spaces. Indeed, a map of spaces $f\colon \X\to\Y$ is fully faithful when each induced map of path spaces $\X(x_0,x_1)\to\Y(fx_0,fx_1)$ is an equivalence; assembled over all pairs $(x_0,x_1)$, this says that the canonical map $\X\simeq\Orb_\X\to\X\times_\Y\X$ is an equivalence over $\X\times\X$, that is, that $f$ is a monomorphism in $\Sp$. The internal rendering of full faithfulness of the Yoneda map is therefore the following statement, the main result of this paper. This will be proved in \cref{subsection: proof of Yoneda}.
\begin{theorem}[Yoneda Embedding]\label{Yoneda Embedding}
    Let $\C$ be a finitely complete locally Cartesian closed $\infty$-category, $X$ an object, and $\U$ a universe classifying the diagonal $\Delta_X\colon X\to X\times X$. The Yoneda map \[\yo_X\colon X\to\U^X\] is a monomorphism.
\end{theorem}

\subsection{Representable fibrations}\label{subsection: representable fibrations}
To run the classical strategy of the proof, we will need the notion of representable fibration in a parametrized form, over an auxiliary object $T$. The need for the parameter is explained in \cref{subsection: yoneda lemma}. By analogy with the external slice construction (Example~\ref{examples: fibrations 2}), we define:
\begin{definition}\label{definition: internal representable presheaf and fibers}
    Let $\C$ be a finitely complete $\infty$-category, $X$ an object, and $\U$ a universe classifying the diagonal $\Delta_X$. Let $x\colon T\to X$ be a generalized element, and $P\to X\times T$ a map.
    \begin{enumerate}
        \item The \textit{internal representable presheaf} $X(-,x)\colon X\times T\to\U$, together with its \textit{internal representable fibration} $\slice{X}{x}\to X\times T$, is defined by the following composition:
\[\begin{tikzcd}[column sep=huge]
    {\slice{X}{x}} \arrow[r] \arrow[d] \arrow[dr, phantom, "\lrcorner"{anchor=center, pos=0.125}] & X \arrow[r] \arrow[d, "\Delta"'] \arrow[dr, phantom, "\lrcorner"{anchor=center, pos=0.125}] & {\U_\ast} \arrow[d, "p"] \\
    {X \times T} \arrow[r, "{X \times x}"] \arrow[rr, "{X(-,x)}"', bend right=18] & {X \times X} \arrow[r, "{X(-,-)}"] & \U
\end{tikzcd}.\]
\item The \textit{fiber} $P_{x}\to T$ of $P$ at $x$ is the following pullback:
\[\begin{tikzcd}[column sep=large]
    {P_{x}} \arrow[r] \arrow[d] \arrow[dr, phantom, "\lrcorner"{anchor=center, pos=0.125}] & P \arrow[d] \\
    T \arrow[r, "{(x,1_T)}"'] & {X \times T}
\end{tikzcd}.\]
\end{enumerate}
\end{definition}
\begin{remark}\label{Notation: Path space between two point in X}
    Given two generalized points $x_0,x_1\colon T\to X$, the fiber of the internal representable fibration $\slice{X}{x_1}$ at $x_0$ is the path object $X(x_0,x_1)\in\slice{\C}{T}$:
\[\begin{tikzcd}
	X(x_0,x_1) \arrow[r] \arrow[d] \arrow[dr, phantom, "\lrcorner", very near start] & \slice{X}{x_1}\arrow[dr, phantom, "\lrcorner", very near start] \arrow[r] \arrow[d] & X \arrow[d] \\
	T \arrow[r, "{(x_0,1_T)}"'] & X \times T \arrow[r, "{1_X \times x_1}"'] & X \times X
\end{tikzcd}.\]
\end{remark}
\begin{remark}\label{remark: Fibered representable is T}
     Using \prelref{item: pullback of a product with X}, we observe that the fibered representable $\slice{X}{x}$ is, as an object over $X\times T$, equivalent to $T$ equipped with the structure map $(x,1_T)\colon T\to X\times T$. This is the internal incarnation of the contractibility of the slice over a final object: if $x\in X$ is a point of a homotopy type $X$, the object $\slice{X}{x}$ is the space of paths in $X$ with one endpoint fixed at $x$, which is contractible by retracting each path to its endpoint.
\end{remark}
\subsection{The internal Yoneda lemma}\label{subsection: yoneda lemma}
Before stating the internal Yoneda lemma, let us explain what its correct formulation should be; the naive translation of the classical statement fails in an instructive way. Classically, for a presheaf $P$ on an $\infty$-groupoid $\X$ and a point $x_0\in \X$, the Yoneda lemma provides an equivalence $\mathrm{Nat}\big(\X(-,x_0),P\big)\simeq P(x_0)$. Translating naively in a locally Cartesian closed $\infty$-category $\C$, for an internal presheaf $\corner{P}\colon X\to\U$ with fiber $P_{x_0}$ at a global point $x_0\colon \ast\to X$, one might guess
\[\U^X\big(\corner{X(-,x_0)},\corner{P}\big)\overset{?}{\simeq} P_{x_0}.\]

This cannot be right: the left hand side is a path object in $\U^X$ \eqref{definition: path objects}, and paths are invertible, so it is symmetric in its two arguments, i.e. it is also equivalent to $\U^X\big(\corner{P},\corner{X(-,x_0)}\big)$, while the right hand side is not. The asymmetry is explained by univalence: paths in the universe correspond to \emph{equivalences} of families, not to arbitrary fibrewise maps. This is made precise by \cref{Lemma: path in fibered universe are equivalences over the base} below, which computes paths in the presheaf object $\U^X$: $\U^X(\corner{P},\corner{Q})\simeq \prod_X\Eq_{/X}(P,Q)$. The object of \emph{all} internal natural transformations is therefore not visible as a path object in $\U^X$; to capture it we must move out of the universe and form the internal object of fibrewise maps directly, using the corresponding internal fibrations. This leads to the correct formulation:
\[\prod_X\Map_{/X}(\slice{X}{x_0},P)\simeq P_{x_0},\]
where $\prod_X\colon\slice{\C}{X}\to\C$ is the dependent product. We notice that this equivalence no longer depends on universes, and its proof (\cref{Yoneda Lemma}) does not either: it holds in any locally Cartesian closed $\infty$-category with pullbacks. $P$ is accordingly allowed to be an arbitrary map in the general statement.

\medskip

A second adjustment is needed: global points do not suffice. While a homotopy type $X$ is recovered from its space of global points, $\Sp(\ast,X)\simeq X$, an object of an $\infty$-category $\C$ is in general not determined by $\C(\ast,X)$; it is instead characterized by the homotopy types of generalized points $\C(T,X)$, for all objects $T\in\C$. To reflect this internal perspective, and because the proof of \cref{Yoneda Embedding} will require it when we apply the lemma to the universal pair of points $\pi_0,\pi_1\colon X\times X\to X$, we formulate the Yoneda lemma in its parametrized form, over an arbitrary base object $T$.

\begin{lemma}[Internal Yoneda Lemma]\label{Yoneda Lemma}
Let $\C$ be a locally Cartesian closed $\infty$-category with pullbacks. Let $X$ be an object of $\C$, $x_0\colon T\to X$ a generalized element, and $P\to X\times T$ any map. There is an equivalence in $\slice{\C}{T}$: 
\begin{equation*}
    \prod_{X\times T\to T}\Map_{/X\times T}(\slice{X}{x_0}, P)\simeq P_{x_0}.
\end{equation*}
\end{lemma}
\begin{proof}
    Throughout, $\pi\colon X\times T\to T$ denotes the projection. By Proposition \ref{detection: objects}, it is enough to show that 
    \begin{equation*}
        \slice{\C}{T}(A,\prod_{X\times T\to T}\Map_{/X\times T}(\slice{X}{x_0}, P) )\simeq \slice{\C}{T}(A,P_{x_0})
    \end{equation*}
    naturally in $A\in\slice{\C}{T}$. This is achieved by the following chain of natural equivalences:
    \begin{align*}
    \slice{\C}{T}\Big(A,\prod_{X\times T\to T}&\Map_{/X\times T}(\slice{X}{x_0}, P)\Big)\\
        &\simeq \slice{\C}{X\times T}\Big(X\times A,\Map_{/X\times T}(\slice{X}{x_0}, P)\Big)&&\proofstep{$\pi^\ast\dashv\Pi_{\pi}$}\\
        &\simeq \slice{\C}{X\times T}\Big((X\times A)\times_{(X\times T)}\slice{X}{x_0},P\Big)&&\proofstep{$-\times_{X\times T}\slice{X}{x_0}\dashv\Map_{/X\times T}(\slice{X}{x_0},-)$}\\
        &\simeq \slice{\C}{X\times T}\Big( A\times_{ T}\slice{X}{x_0},P\Big) &&\proofstep{by \prelref{item: pullback of a product with X}}\\
        &\simeq\slice{\C}{X\times T}\big( A,P\big) &&\proofstep{$\slice{X}{x_0}\simeq T$ over $X\times T$, by \cref{remark: Fibered representable is T}}\\
        &\simeq \slice{\C}{T}\big(A,P\times_{(X\times T)}T\big) &&\proofstep{\prelref{item: Hom C vs slice} \& \prelref{item: Slice of slice} with $T\xrightarrow{(x_0,1_T)} X\times T$}  \\
        &\simeq  \slice{\C}{T}\big(A,P_{x_0}\big) &&\proofstep{by \cref{definition: internal representable presheaf and fibers}}
    \end{align*}
\end{proof}

\subsection{Proof of the Yoneda embedding theorem}\label{subsection: proof of Yoneda}
We now apply the internal Yoneda lemma to internal representable fibrations. Two further lemmas are needed. The first asserts that all internal natural transformations between representables are invertible, as expected from the classical case, where they correspond to paths in $X$. The second asserts that path objects in $\U^X$ compute internal objects of fibrewise equivalences, giving an exponentiated form of univalence.

\begin{lemma}\label{lemma: nat trans between representables are nat equivalences}
    Let $\C$ be a finitely complete locally Cartesian closed $\infty$-category, and let $x_0,x_1\colon T\to X$ be generalized elements. There is an equivalence in $\slice{\C}{T}$: \[\prod_{X\times T\to T}\Eq_{/X\times T}(\slice{X}{x_0},\slice{X}{x_1})\simeq \prod_{X\times T\to T}\Map_{/X\times T}(\slice{X}{x_0},\slice{X}{x_1}).\]
\end{lemma}
\begin{proof}
By \cref{prop: an equivalence that factors with a mono is an equivalence}, it is enough to show that the map defined by the Yoneda lemma
\begin{equation}\label{eq: Yoneda equivalence for representables}
X(x_0,x_1)\xrightarrow{\simeq} \prod_{X\times T\to T}\Map_{/X\times T}(\slice{X}{x_0},\slice{X}{x_1})
\end{equation}
over $T$ factors through the monomorphism \[\prod_{X\times T\to T}\Eq_{/X\times T}(\slice{X}{x_0},\slice{X}{x_1})\hookrightarrow \prod_{X\times T\to T}\Map_{/X\times T}(\slice{X}{x_0},\slice{X}{x_1}).\]
This is indeed a monomorphism since $\prod_{X\times T\to T}$ preserves limits, hence monomorphisms.
This happens if and only if the adjoint of \eqref{eq: Yoneda equivalence for representables}
\begin{equation}\label{eq: Adjoint of Yoneda eq for representables}
    X\times X(x_0,x_1)\to \Map_{/X\times T}(\slice{X}{x_0},\slice{X}{x_1})\qquad \text{ over }X\times T
\end{equation}
factors through $\Eq_{/X\times T}(\slice{X}{x_0},\slice{X}{x_1})$. By definition of the object of equivalences (\cref{subsection: equivalences}), this happens precisely if the adjoint of \eqref{eq: Adjoint of Yoneda eq for representables} 
\begin{equation*}
    \big(X\times X(x_0,x_1)\big)\times _{(X\times T)} \slice{X}{x_0}\xrightarrow{}\big(X\times X(x_0,x_1)\big)\times_{(X\times T)}\slice{X}{x_1}
\end{equation*}
is an equivalence over $X\times X(x_0,x_1)$. But using \prelref{item: pullback of a product with X} the latter is equivalent to \begin{equation*}
    X(x_0,x_1)\times _T \slice{X}{x_0}\to X(x_0,x_1)\times _T\slice{X}{x_1},
\end{equation*}
which by \cref{remark: Fibered representable is T} is precisely a map \[X(x_0,x_1)\to X(x_0,x_1)\]
over $X(x_0,x_1)$. We deduce that this map is the identity, hence an equivalence, as desired.
\end{proof}
\begin{lemma}\label{Lemma: path in fibered universe are equivalences over the base}
    Let $\C$ be a finitely complete locally Cartesian closed $\infty$-category, $X\in\C$ an object, and $\U$ a univalent universe classifying the diagonal $\Delta_X$. Suppose that $\corner{P},\corner{Q}\colon X\times T\to\U$ are two internal presheaves with corresponding internal fibrations $P$ and $Q$ over $X\times T$. Then there is an equivalence in $\slice{\C}{T}$:
     \[\U^X\Big(\corner{P},\corner{Q}\Big)\simeq \prod_{X\times T\to T}\Eq_{/X\times T}(P,Q),\]
     where we keep the notation $\corner{P},\corner{Q}$ for their adjoints $T\to \U^X$.
\end{lemma}
\begin{proof}
    By \cite[§1, Item (19)]{Rasekh_TheoryElementaryHigherToposes}, the map $p\times X\colon\U_\ast\times X\to \U\times X$ is univalent in the slice over $X$. By univalence (\cref{Univalence Axiom}) applied in $\slice{\C}{X}$, there is a Cartesian square in $\slice{\C}{X}$:
\[\begin{tikzcd}[column sep=large]
    {\Eq_{/X \times T}(P,Q)} \arrow[r] \arrow[d] \arrow[dr, phantom, "\lrcorner"{anchor=center, pos=0.125}] & {\U \times X} \arrow[d, "\Delta"] \\
    {X\times T} \arrow[r, "{\big((\corner{P},\pi_1), (\corner{Q}, \pi_1)\big)}"'] & {(\U \times X) \times_X (\U \times X)}
\end{tikzcd}.\]
Since the left hand side lives over $T$, one can take the Cartesian product with $T$ on the right hand side, yielding a Cartesian square in $\slice{\C}{X\times T}$ by \prelref{item: pullback of a product with X}. One can now apply the limit preserving functor $\prod_{X\times T\to T}\colon \slice{\C}{X\times T}\to\slice{\C}{T}$ to obtain the following Cartesian square over $T$:
\begin{equation}\label{Cartesian square: path space in U^X}
    \begin{tikzcd}[column sep=large]
    {\prod\limits_{X\times T\to T}\Eq_{/X\times T}(P,Q)} \arrow[r] \arrow[d] \arrow[dr, phantom, "\lrcorner"{anchor=center, pos=0.125}] & {\U^X\times T} \arrow[d, "{\Delta\times T}"] \\
    T \arrow[r, "{(\corner{P},\corner{Q}, 1_T)}"'] & {(\U^X\times\U^X)\times T}
\end{tikzcd}.
\end{equation}
On the right hand side, we used Beck-Chevalley \cite[Lemma 2.1]{Frey&Rasekh_ConstructingCoproductsInLCCC} for the Cartesian square
\[\begin{tikzcd}
    {X \times T} \arrow[r] \arrow[d] \arrow[dr, phantom, "\lrcorner"{anchor=center, pos=0.125}] & X \arrow[d] \\
    T \arrow[r] & \ast
\end{tikzcd}\]
to obtain an equivalence \[\prod_{X\times T\to T} (-\times T)\simeq (\prod_X-)\times T\] between functors $\slice{\C}{X}\to\slice{\C}{T}$, as well as the equivalence $\prod_X (A\times X)\simeq A^X$ for all $A$. Using again \prelref{item: pullback of a product with X}, one can forget about the $T$'s on the right hand side of \eqref{Cartesian square: path space in U^X} to obtain a pullback square 
\[\begin{tikzcd}[column sep=large]
    {\prod_{X \times T \to T} \Eq_{/X \times T}(P,Q)} \arrow[r] \arrow[d] \arrow[dr, phantom, "\lrcorner"{anchor=center, pos=0.125}] & {\U^X} \arrow[d, "\Delta"] \\
    T \arrow[r, "{(\corner{P}, \corner{Q})}"'] & {\U^X \times \U^X}
\end{tikzcd}.\]
By definition of path objects \eqref{definition: path objects}, this pullback is precisely $\U^X\Big(\corner{P},\corner{Q}\Big)$ as desired.
\end{proof}

\begin{proof}[Proof of \cref{Yoneda Embedding}]
    By Proposition \ref{prop: characterization monomorphisms 5}, it suffices to show that the commutative square
    \begin{equation}\label{proof: Yoneda embedding commutative square}
    \begin{tikzcd}[ampersand replacement=\&, column sep=huge]
        X \& {\U^X} \\
        {X\times X} \& {\U^X\times\U^X}
        \arrow["{\yo_X}", from=1-1, to=1-2]
        \arrow["{\Delta_X}"', from=1-1, to=2-1]
        \arrow["{\Delta_{\U^X}}"', from=1-2, to=2-2]
        \arrow["{\yo_X\times \yo_X}", from=2-1, to=2-2]
    \end{tikzcd}
    \end{equation}
    is Cartesian. We first produce, for every pair of generalized elements $x_0,x_1\colon T\to X$, an equivalence between the generalized fibers of $\Delta_X$ at $(x_0,x_1)\colon T\to X\times X$ and of $\Delta_{\U^X}$ at \[\big(X(-,x_0),X(-,x_1)\big)\colon T\to X\times X\to\U^X\times\U^X.\] In more familiar words, we show that there is an equivalence in $\slice{\C}{T}$: 
\begin{equation}\label{proof: Yoneda embedding eq}
    X(x_0,x_1)\simeq\U^X\big(X(-,x_0), X(-,x_1)\big).
\end{equation}
    The former object is $(\slice{X}{x_1})_{x_0}$ over $T$ by \cref{Notation: Path space between two point in X}, which we identify in $\slice{\C}{T}$ with:
   \begin{align*}
    &\prod_{X\times T\to T}\Map_{/X\times T}(\slice{X}{x_0}, \slice{X}{x_1})&&\proofstep{by the Yoneda \cref{Yoneda Lemma}}\\
   \simeq & \prod_{X\times T\to T}\Eq_{/X\times T}(\slice{X}{x_0}, \slice{X}{x_1})&&\proofstep{by \cref{lemma: nat trans between representables are nat equivalences}}\\
   \simeq & \,\,\U^X\big( X(-,x_0), X(-,x_1)\big)&&\proofstep{by \cref{Lemma: path in fibered universe are equivalences over the base}.}
\end{align*}

    We emphasize that this equivalence is not claimed to be the comparison map between the fibers of \eqref{proof: Yoneda embedding commutative square} induced by $\yo_X$. Instead of showing that this is indeed the case, we use it to argue that the comparison map \[\varphi\colon X\to (X\times X)\times_{(\U^X\times\U^X)}\U^X\] induced by \eqref{proof: Yoneda embedding commutative square} is an equivalence, which shows that the square is Cartesian. To this end, specialize to $T=X\times X$, with the two projections $(x_0,x_1)=(\pi_1,\pi_2)=1_{X\times X}\colon T\to X\times X$ as generalized elements. Over $T$ we now have
    \begin{itemize}
        \item  an equivalence $X(x_0,x_1)\simeq X$ with structure map $\Delta_X\colon X\to X\times X$. This map admits a retraction $\pi_1$;
        \item an identification of $\U^X\big( X(-,x_0), X(-,x_1)\big)$ as the pullback $(X\times X)\times_{(\U^X\times\U^X)}\U^X$ of \eqref{proof: Yoneda embedding commutative square}.
        \end{itemize}
    Thus \eqref{proof: Yoneda embedding eq} is an equivalence $X\xrightarrow{\sim}(X\times X)\times_{(\U^X\times\U^X)}\U^X$ over $X\times X$, and the comparison map $\varphi$ is another morphism between the same two objects over $X\times X$. Since $\Delta_X$ has a retraction, every such morphism is an equivalence by \cref{lemma: retract rigidity}. In particular $\varphi$ is an equivalence.
\end{proof}
\section{Applications}\label{section: application}
We record two applications. The first specializes the main results to the $\infty$-category of spaces, recovering the classical Yoneda statements for $\infty$-groupoids.
The second is the application that motivated this work, to the theory of covering spaces in $\infty$-topoi.
\subsection{Yoneda for $\infty$-groupoids}
The $\infty$-category $\Sp$ is finitely complete and locally Cartesian closed, and for every small $\infty$-groupoid $X$ the diagonal $\Delta_X$ is classified by an object classifier $\Sp^\simeq_{\kappa,\ast}\to\Sp^\simeq_\kappa$, for $\kappa$ a sufficiently large small regular cardinal \cite[Ex.~3.5, Prop.~3.8]{GepnerKock_UnivalenceInLCCC}. The main results of \cref{section : Yoneda} therefore specialize to $\Sp$, where they recover the classical statements.
\begin{corollary}\label{yoneda lemma for groupoids}
    Let $\X$ be a small $\infty$-groupoid, $F\colon \X^\op\to\Sp$ a presheaf, and $x$ an object in $\X$. Then there is an equivalence \[F(x)\simeq \Psh(\X)\left(\X(-,x),F\right).\]
\end{corollary}
\begin{proof}
    By the internal Yoneda \cref{Yoneda Lemma} with $\C=\Sp$, $T=\ast$, and $P\to\X$ the unstraightening of $F$, we have 
    \begin{align*}
        F(x)&\simeq P_{x}\simeq \prod_{\X}\Map_{/\X}\left(\slice{\X}{x},P\right)\\
        &\simeq \Sp(\ast, \prod_{\X}\Map_{/\X}\left(\slice{\X}{x},P\right))\\
        &\simeq \slice{\Sp}{\X}\left(\slice{\X}{x},P\right)&&\proofstep{$\left(-\times \X\dashv \prod_{\X}\right)$ and $\left(-\times_{\X} \slice{\X}{x}\dashv\Map_{/\X}(\slice{\X}{x},-)\right)$}\\
        &\simeq \mathscr{R}\mathrm{Fib}(\X)\left(\slice{\X}{x},P\right)&&\proofstep{$\Sigma_x\colon \slice{\X}{x}\to\X$ and $P\to \X$ are right fibrations}\\
        &\simeq \Psh(\X)\left(\X(-,x),F\right) &&\proofstep{by straightening \eqref{eq: (un)straightening} and Example \ref{examples: fibrations 2}.}
    \end{align*}
\end{proof}
\begin{corollary}\label{Cor: yoneda for groupoids}
    Let $\X$ be a small $\infty$-groupoid. Then the Yoneda map $\yo_\X\colon\X\to \Psh(\X)$ is fully faithful.
\end{corollary}
\begin{proof}
    By the internal Yoneda embedding \cref{Yoneda Embedding}, the Yoneda map \begin{equation}\label{Cor: yoneda for groupoids eq}
        \yo_\X\colon \X\longrightarrow \U^\X=\Map_\Sp(\X,\Sp_\kappa^\simeq)\simeq\Fun(\X,\Sp_\kappa)^\simeq
    \end{equation}
    is a monomorphism of spaces, that is, fully faithful. Since $\X\simeq \X^\op$ and $\Sp_\kappa^\simeq\hookrightarrow\Sp^\simeq$ is fully faithful, the target of \eqref{Cor: yoneda for groupoids eq} embeds in $\Psh(\X)^\simeq$. Finally, every natural transformation between representable presheaves on a groupoid is invertible by \cref{lemma: nat trans between representables are nat equivalences} for $\C=\Sp$ and $T=\ast$. Whence \[\Psh(\X)^\simeq(\X(-,x_0),\X(-,x_1))\simeq \Psh(\X)(\X(-,x_0),\X(-,x_1))\] for all $x_0,x_1\in \X$. The canonical comparison
    $\X(x_0,x_1)\to\Psh(\X)(\X(-,x_0),\X(-,x_1))$ is thus an equivalence, i.e.\ $\yo\colon \X\to\Psh(\X)$
    is fully faithful.
    \end{proof}
    \subsection{Covering spaces}\label{subsection: application covering spaces}
This work originated in the theory of $n$-covering spaces in an $\infty$-topos $\E$ developed in \cite{Constantin_HigherCoveringSpacesinHigherTopos}, whose central theorem reads as follows:
\begin{theorem*}[{\cite[Cor.~5.34]{Constantin_HigherCoveringSpacesinHigherTopos}}]
    Let $(X,x)\in\E_\ast$ be a pointed connected object, and let $p\colon E\to X$ be a normal $n$-covering map. Then $\B^{n-1}\pi_n(E,e)\to \Pi_n(X,x)$ is a normal $n$-group morphism, and there is an equivalence of $n$-groups:
    \[\Deck(p)\simeq \Pi_n(X,x)\sslash \B^{n-1}\pi_n(E,e).\]
\end{theorem*}
Here $p$ is an $(n-1)$-truncated map, $E$ is pointed and $(n-1)$-connected, and $\pi_n(E,e)\leq\pi_n(X,x)$ is normal. The $n$-group of deck transformations of $p$ is $\Deck(p)\coloneqq \prod_X\Eq_{/X}(E,E)\in\Grp(\E)$, and the fundamental $n$-group is $\Pi_n(X,x)\coloneq \Omega(\tau_n X)\in\Grp(\E)$. We indicate how the Yoneda embedding enters, in the representative case $n=1$ with $p$ the universal $1$-cover (so $\pi_1(E,e)=0$ and $\Pi_1(X,x)\simeq\pi_1(X,x)$).
\begin{proof}[Proof sketch]
 Let $\U$ be a universe classifying $(n-1)$-truncated maps. Since $\B\pi_1(X,x)\simeq\tau_1X$ is $1$-truncated, its diagonal is $0$-truncated, hence classified by $\U$. Transposing the composite
\[X\times X\xrightarrow{\eta\times\eta}\B\pi_1(X,x)\times\B\pi_1(X,x)\xrightarrow{\B\pi_1(X,x)(-,-)}\U\] 
  and precomposing with the base point gives a commutative diagram
\[\begin{tikzcd}
	\ast \arrow[r, "x"] \arrow[dr, two heads] & X \arrow[r] \arrow[d, "\eta"'] & {\U^X} \\
	& {\B\pi_1(X,x)} \arrow[r, "\yo"', hook] & {\U^{\B\pi_1(X,x)}} \arrow[u, "\U^\eta"', "\simeq"]
\end{tikzcd}.\]
The base point is an effective epimorphism, $\yo$ is a monomorphism by \cref{Yoneda Embedding}, and $\U^\eta$ is an equivalence because $\eta$ is $1$-connected and $\U$ is $1$-truncated. Hence $\B\pi_1(X,x)$ is the image of the point $\ast\to\U^X$. This point classifies the universal $1$-cover, picking out the internal presheaf $y\mapsto\pi_0(X(x,y))$. Its image is therefore a pointed connected subobject of $\U^X$, whose loop object is \[\Omega_{\corner{E}}\,\U^X\simeq\prod_X\Eq_{/X}(E,E)=\Deck(p),\] by \cref{Lemma: path in fibered universe are equivalences over the base}. A pointed connected object is the delooping of its loops, so the image is $\B\Deck(p)$. Comparing the two computations of the image gives  $\B\Deck(p)\simeq\B\pi_1(X,x)$, and looping yields $\Deck(p)\simeq\pi_1(X,x)$, the stated equivalence in this case.
\end{proof}
    
\printbibliography
\end{document}